\documentclass[11pt]{amsart}

\usepackage{enumerate,url,amssymb,upref}

\newtheorem{theorem}{Theorem}
\newtheorem{lemma}[theorem]{Lemma}
\newtheorem{proposition}[theorem]{Proposition}

\theoremstyle{definition}

\newtheorem{example}[theorem]{Example}
\newtheorem{corollary}[theorem]{Corollary}

\theoremstyle{remark}
\newtheorem{remark}[theorem]{Remark}

\newtheorem*{conjecture}{Conjecture}

\newcommand{\abs}[1]{\left\lvert#1\right\rvert}
\newcommand{\norm}[1]{\left\lVert#1\right\rVert}
\newcommand{\inn}[1]{\left\langle#1\right\rangle}
\newcommand{\C}{\mathbb{C}}
\newcommand{\M}{\mathcal{M}}
\newcommand{\R}{\mathbb{R}}

\newcommand{\dtext}{\textnormal d}
\DeclareMathOperator{\dist}{dist}
\DeclareMathOperator{\re}{Re}
\DeclareMathOperator{\im}{Im}
\DeclareMathOperator{\Spec}{Spec}
\DeclareMathOperator{\esssup}{ess\,sup}

\def\XXint#1#2#3{{\setbox0=\hbox{$#1{#2#3}{\int}$}
\vcenter{\hbox{$#2#3$}}\kern-.5\wd0}}
\newcommand{\loc}{\textnormal{loc}}

\def\le{\leqslant}
\def\ge{\geqslant}

\begin{document}

\title{On Invertibility of Sobolev Mappings}

\author{Leonid V. Kovalev}
\address{Department of Mathematics, Syracuse University, Syracuse,
NY 13244, USA}
\email{lvkovale@syr.edu}
\thanks{Kovalev was supported by the NSF grant DMS-0700549.}

\author{Jani Onninen}
\address{Department of Mathematics, Syracuse University, Syracuse,
NY 13244, USA}
\email{jkonnine@syr.edu}
\thanks{Onninen was supported by the NSF grant  DMS-0701059.}


\date{December 9, 2008}


\begin{abstract} We prove local and global invertibility of Sobolev solutions of certain differential inclusions which
prevent the differential matrix from having negative eigenvalues.
Our results are new even for quasiregular mappings in two dimensions.
\end{abstract}

\maketitle

\section{Introduction}

We study differential inclusions under which a Sobolev mapping  $f \in W^{1,p}_{\loc}(\Omega ; \R^n)$ is invertible, at least locally. Here $\Omega \subset \R^n$ is a connected open set, $n \ge 2$.  A  differential inclusion means that the differential matrix $Df$, which exist a.e.,  must  be an element of a certain  subset of $n\times n$ matrices. Specifically, we consider the sets
\[\mathcal M_n(\delta) = \left\{ A \in \R^{n \times n} \colon \inn{A \xi , \xi} \ge \delta \abs{A\xi}\abs{\xi} \quad \textnormal{for all } \xi \in \R^n  \right\} \]
where $-1\le \delta \le 1$.
Note that $\M_n (\delta_1) \subset \M_n (\delta_2)$ if $\delta_1 > \delta_2$, and  $\mathcal M_n(-1) $ consists of all $n \times n$ matrices. When $\delta \le 0$, the differential inclusion $Df \in \M_n(\delta) $ admits noninvertible solutions such as orthogonal projections.
For this reason we   introduce a subset of $\M_n (\delta)$ which contains matrices with distortion at most $K$, $1\le K < \infty$.
\[\M_n(\delta, K) =    \left\{ A \in  \mathcal M_n(\delta)  \colon \norm{A}^n \le K \det A \right\} . \]
  Here and in what follows $\norm{A}$ is the operator norm of matrix $A$ with respect to the Euclidean vector norm, which is denoted by $\abs{\cdot}$.

\begin{theorem}\label{ThMain0}
Let $f \in W^{1,n}_{\loc} (\Omega ; \R^n)$ be a nonconstant  mapping.   If there exist $\delta >-1$ and $K< \infty$ such that the differential matrix $Df(x)$ belongs to $ \mathcal M_n(\delta , K )$  for almost every $x\in \Omega$,
then $f$ is a local homeomorphism. If in addition $\Omega =\R^n$, then $f$ is a homeomorphism.
\end{theorem}

A mapping $f\in W^{1,n}_{\loc}(\Omega ; \R^n)$ such that $\norm{Df}^n \le K \det Df$ a.e. in $\Omega$ is called $K$-quasiregular \cite{Rebook, Ribook}. It is called $K$-quasiconformal if $f$   is also a homeomorphism. In the planar case  $n=2$ it is convenient to identify $\R^2$ with $\C$ and use the complex derivatives $f_z, f_{\bar z}$. Set $\tau_K= 2 \sqrt{K}/(K+1)$ for $K \ge 1$.

\begin{corollary}\label{QRThm}   Suppose that  $f \colon \Omega \to \mathbb C$ is $K$-quasiregular and nonconstant.   If  $\re f_z \ge - \tau \abs{f_z}$  a.e. for some $\tau< \tau_K$, then $f$ is a local homeomorphism. If in addition $\Omega = \C$, then $f$ is a homeomorphism.
\end{corollary}

Corollary \ref{QRThm} extends the main results of~\cite{IKO}, where it was assumed that $\re f_z\ge 0$ a.e.
Its sharpness is  demonstrated by the following example.

\begin{example}\label{branchex}
For every $1 \le K < \infty $ there exists a nonconstant  $K$-quasi\-regular mapping $f\colon\C\to\C$ such that
$\re f_z  \ge  - \tau_K \left| f_z\right|$ and $f$ is not a local homeomorphism.
\end{example}

Theorem \ref{ThMain0} is a special case of more general results stated as Theorems \ref{ThMain} and \ref{ThMain2} in which the assumption of quasiregularity is relaxed.   We say that   $f\in W^{1,n}_{\loc}(\Omega, \R^n)$ is a mapping of   finite distortion if  $\det Df(x) \ge 0$ a.e. and the matrix $Df$ vanishes   a.e. in the zero set of $\det Df$.  With such mappings  we associate two distortion function, outer and inner.
 \begin{equation}\label{distclassical}
 \begin{split}
{K}_O (x,f)=\begin{cases}\frac{\norm{Df(x)}^n}{J(x,f)}, &  \;  J(x,f)>0 \\ 1, & \; J(x,f)=0
\end{cases} \\
K_I(x,f) = \begin{cases}\frac{\norm{D^\sharp f(x)}^n}{J(x,f)^{n-1}} , & \;   J(x,f)>0 \\ 1, & \; J(x,f)=0 .
\end{cases}
\end{split}
\end{equation}
Here $D^\sharp f$ stands for the cofactor matrix of $Df$ and    $J(x,f)=\det Df(x)$.  A standard reference for mappings of finite distortion is the monograph by Iwaniec and Martin \cite{IMbook}.
When $\esssup K_O(\cdot,f)=K<\infty$, i.e., $f$ has bounded distortion,  we recover the class of  $K$-quasiregular mappings.

\begin{theorem}\label{ThMain}
Suppose that $f \in W^{1,n}_{\loc} (\Omega; \R^n)$ is nonconstant and $K_O(\cdot ,f) \in L^p_{\loc} (\Omega)$ where $p=1$ if $n=2$ and $p>n-1$ if $n \ge 3$.  If there exists $\delta >-1$ such that $Df(x) \in \mathcal M_n(\delta)$  for almost every $x\in \Omega$,
then $f$ is a local homeomorphism.
\end{theorem}

\begin{theorem}\label{ThMain2}
Under the assumptions of Theorem \ref{ThMain}, if $\Omega = \R^n$, then $f$ is a global homeomorphism onto $\R^n$.
\end{theorem}

One cannot allow $p<n-1$ in Theorem \ref{ThMain}, as is demonstrated by Example~\ref{ball}. Remarks \ref{monotone}  and \ref{notmonotone} in Section \ref{mainproofs} show that Theorem \ref{ThMain2} does not hold when  $\Omega$ is a proper convex domain (even a half-space) unless $\delta \ge 0$.

The assumptions on the distortion  in Theorems~\ref{ThMain} and~\ref{ThMain2} guarantee that $f$ is a discrete and open mapping, see Theorem~\ref{discreteopen}. If we know a priori that $f$ is discrete and open, then a sufficiently high degree of integrability  of $Df$ can substitute for the integrability of $K_O(\cdot, f)$.

\begin{theorem}\label{ThMain3}
Let  $f \in W^{1,p}_{\loc} (\Omega; \R^n)$, where $p=2$ if $n=2$ and  $p>(n-1)^2$ if $n \ge 3$. Suppose further that $f$ is discrete and open and  there exists $\delta >-1$ such that  $Df(x) \in  \mathcal M_n(\delta)$  for a.e.  $x\in \Omega$. Then $f$ is a local homeomorphism.  If in addition $\Omega = \R^n$, then $f$ is a homeomorphism.
\end{theorem}

We conclude the introduction by comparing our results with known injectivity  theorems. It is easy to see that $A \in \M_n(\delta)$ for some $\delta>-1$ if and only if $A$ does not have negative  eigenvalues; that is, $\Spec (A) \cap (-\infty, 0) =\varnothing$. Injectivity of differentiable mappings under similar spectral conditions has been studied recently, e.g., in~\cite{FGR,SX} and references therein. In  \cite{FGR} it is proved that if $f$ is differentiable everywhere and $\Spec (Df(x)) \cap (- \epsilon, 0] =\varnothing$ for some $\epsilon >0$ independent of $x$, then $f$ is injective. Under these assumptions $f$ is already known to be a local homeomorphism.
In contrast, we consider weakly differentiable mappings, not excluding the possibility $Df=0$. For us, proving that $f$ is a local homeomorphism
is the main part of the argument.
The invertibility of Sobolev mappings was also studied in~\cite{Ba} in connection with boundary value problems of nonlinear elastostatics.

Quasiregular mappings in dimensions $n\ge 3$ are known to be locally invertible if the distortion $K$ is sufficiently small \cite{Go, MRV}. A quantitative version  of this result was recently proved  by  Rajala \cite{Ra2}.  It is a 
long-standing open problem whether $K_I<2$ suffices \cite{MRV}.  Another way to obtain injectivity is to impose a stronger regularity requirement on $f$.  Heinonen and Kilpel\"ainen \cite{HK} proved that  any mapping  $f\in W^{2,2}_{\loc}(\Omega ; \mathbb R^n) \cap W^{1, \infty}_{\loc} (\Omega ; \mathbb R^n) $ with $J(\cdot , f ) \ge c >0$  is  a local homeomorphism.  They raised the interesting  question
whether a quasiregular mapping $f$ with dilatation tensor $G_f\in W^{1,2}_{\loc}$, where
\[G_f = J(\cdot , f)^{-2/n} (Df)^tDf, \]
is a local homeomorphism. 

Finally, the injectivity of planar quasiregular mappings under certain boundary conditions was established in~\cite{LN}, \cite{BDIS} and already found applications in the theory of differential inclusions~\cite{FS}.

\section{Preliminaries}\label{prelim}
We collect the basic properties of discrete and open mappings which can be found in \cite[\S I.4]{Ribook}.

Let $\Omega$ be a domain in $\R^n$ and let
 $f \colon \Omega \to \R^n$ be a continuous mapping. We say that $f$ is  discrete if the preimage of every point $y\in f(\Omega)$ is a discrete set.  If $f$ takes open sets to open sets then $f$ is an open mapping. A domain $G \Subset \Omega$ is called a normal domain for $f$ if $f (\partial G)= \partial f(G)$. For $y\in \R^n \setminus  f(\partial G)$ the local degree of $f$ at $y$ with respect to $G$ is denoted $\mu(y,f,G)$. If $x\in G$ and $G$ is a normal  domain such that $G \cap f^{-1}(f(x))=\{x\}$, then $G$ is called a normal neighborhood of $x$. In this case we  write $i(x,f)= \mu (f(x),f,G)$ for the topological index of $f$ at $x$.   The branch set $B_f$ consists of all points in $\Omega$ at which $f$ is not a local homeomorphism. Outside the branch set either $i(x,f) \equiv 1$ or $i(x,f) \equiv -1$.  In the first case $f$ is sense-preserving.   For $E \subset \Omega$ we write
 \[N(f,E)= \sup_{y\in f(E)} \# \{ f^{-1}(y) \cap E \}.  \]
If $f$ is a discrete and open mapping and $E \Subset \Omega$, then $N(f,E)< \infty$. Also, $N(f,G)=i(x,f)$ if $G$ is a normal neighborhood of $x$. The open ball with radius $r$ centered at $a$ will be denoted by $B(a,r)$.

On several occasions we will use the following result about limits of local homeomorphisms. Although it is likely to be known, we could not find it explicitly stated in the literature. We give a proof that was communicated to us by Jussi V\"ais\"al\"a  \cite{Va}.

\begin{proposition}\label{PrVa}
Let $f_j \colon \Omega \to \R^n$ be a sequence of (local) homeomorphisms  which converges locally uniformly to a discrete and open mapping $f$.
Then $f$ is a (local) homeomorphism.
\end{proposition}
\begin{proof}
First we consider the case of local homeomorphisms. For this we fix $x_0 \in \Omega$ and let   $y_0=f(x_0)$, $y_j=f_j(x_0)$.  We will prove that there is a neighborhood of $x_0$ in which $f_j$ is injective for all $j$. Let $U(x_0,f,r)$ be the connected  component of $f^{-1} (B(y_0,r))$ containing  $x_0$. We choose $r$ sufficiently  small so that   $U=U(x_0,f,r)$  is a normal neighborhood of $x_0$ \cite[Lemma I.4.9]{Ribook}.    We may assume that $\abs{f_j(x)-f(x)}< r/3$ for all $x\in \overline{U}$ and for every $j$.

Set   $U_1=U(x_0,f,r/3)$. Since $\abs{y_j-y_0} <r/3$, there is a component $Q_j$ of $f_j^{-1}(B(y_0,2r/3))$ containing $x_0$. For $x\in \partial U$ we have $\abs{f_j(x)-y_0}>2r/3$, therefore $f_j(\partial U) \cap \partial B(y_0,2r/3)= \varnothing$ and thus $\overline{Q}_j \subset U$. By \cite[Lemma 2.2]{MRV}, $f_j$ maps $Q_j$ homeomorphically onto $B(y_0,2r/3)$. As $f(U_1)=B(y_0,r/3)$, we have $f_j(U_1) \subset B(y_0,2r/3)$, which implies $U_1 \subset Q_j$. Hence  the restriction of $f_j$ to ${U_1}$ is injective for all $j$.

Using the homotopy invariance of the topological degree \cite[Prop. I.4.4]{Ribook}, we obtain
\[i(x_0,f) = \mu (x_0,f,U_1)= \mu (x_0,f_j, U_1) = \pm 1\]
for all $j$.
Therefore, $f$ is a local homeomorphism at $x_0$, \cite[Prop. I.4.10]{Ribook}.

Now we turn to the case of global homeomorphisms. Suppose to the contrary that there are points $x_0 \ne x'_0$ such that $f(x_0)=f(x'_0)=y_0$.   We choose $r$ sufficiently  small so that   $U=U(x_0,f,r)$ and $U'=U(x_0',f,r)$  are  disjoint normal neighborhoods of $x_0$ and $x'$ respectively.  Let $Q_j= f_j^{-1}(B(y_0,2r/3))$; this is a connected set because $f_j$ is a homeomorphism. The above argument shows that $Q_j \subset U$ and $Q_j \subset U'$. This is a contradiction because $U$ and $U'$ are disjoint.
\end{proof}

The following sufficient condition for a local homeomorphism to be injective is due to  John  \cite[p.87]{Jo}, see also~\cite{Os}.

\begin{theorem}\label{john}  Suppose that $f \colon \R^n \to \R^n$ is a local homeomorphism and there exists $\epsilon >0$ such that
\[\liminf_{x \to a} \frac{\abs{f(x) -f(a)}}{\abs{x-a}} \ge \epsilon  \qquad \mbox{ for all } a \in \R^n.\]
Then $f$ is a homeomorphism onto $\R^n$.
\end{theorem}

\section{Distortion estimates}\label{distortion}

The following extension of  the fundamental Reshetnyak's theorem \cite{Rebook} is due to  Iwaniec and  \v Sver\'ak \cite{IS}  in the planar case, and to Manfredi and Villamor \cite{MV} in higher dimensions.
\begin{theorem}\label{discreteopen}
Suppose that $f \in W^{1,n}_{\loc} (\Omega, \R^n)$ and $K_O(\cdot ,f) \in L^p_{\loc} (\Omega)$ where $p=1$ if $n=2$ and $p>n-1$ if $n \ge 3$. Then $f$ is either constant or both discrete and open.
\end{theorem}

For our applications of Theorem~\ref{discreteopen} we need to estimate the distortion functions of
$f^\lambda (x):= f(x)+ \lambda x$, $\lambda >0$. This task is of purely algebraic nature. Namely,
for an $n\times n$ matrix $A$ with $\det A>0$ we write
\begin{equation}\label{distdef}
K_O(A)=\frac{\norm{A}^n}{\det A}=\frac{\sigma_1^n}{\sigma_1 \sigma_2 \cdots \sigma_n },\quad
K_I(A)=K_O(A^{-1})=\frac{\sigma_1 \sigma_2 \cdots \sigma_n }{\sigma_n^n}
\end{equation}
where $\sigma_1(A)\ge\dots \ge\sigma_n(A)$ are the singular values of  $A$ arranged in the nonincreasing order.
Note that~\eqref{distdef} is consistent with (\ref{distclassical}), see ~\cite[\S 6.4]{IMbook}. We write $a\wedge b$ for the minimum of $a$ and $b$.

\begin{lemma}\label{lambdadist}
Let $A\in \mathcal M_n(\delta)$, $\delta >-1$, be an invertible matrix.  For $\lambda>0  $ let $A^{\lambda}=A+\lambda I$.
Then
\begin{equation}\label{lambda1}
K_O(A^{\lambda})\le C(\delta,n)K_O(A)\quad \text{and} \quad K_I(A^{\lambda})\le C(\delta,n)K_I(A),
\end{equation}
where $C(\delta,n)=(2/\sqrt{1- (\delta \wedge 0)^2 })^{n-1}$. Furthermore,
\begin{equation}\label{lambda2}
\norm{ (A^\lambda )^{-1}} \le \frac{1}{\lambda \sqrt{1- (\delta\wedge 0)^2 }}.
\end{equation}
\end{lemma}

\begin{proof} We may assume $\delta\in (-1,0]$, in which case $\delta\wedge 0=\delta$.
Observe that if two vectors $u,v\in\R^n$ satisfy $\inn{u,v}\ge \delta\abs{u}\abs{v}$,
then the reverse triangle inequality holds:
\begin{equation}\label{revtri}
\abs{u+v}\ge \sqrt{1-\delta^2}\max\{\abs{u},\abs{v}\}.
\end{equation}
Indeed, we have
\[\abs{u+v}^2\ge \abs{u}^2+2\delta\abs{u}\abs{v}+\abs{v}^2,\]
which yields~\eqref{revtri} upon completing the square in two different ways.

Estimate~\eqref{revtri} and the triangle inequality imply
\begin{equation}\label{lambda3}
\sqrt{1-\delta^2}\max\{\abs{Ax},\lambda\abs{x}\}\le \abs{A^{\lambda}x}\le 2\max\{\abs{Ax},\lambda\abs{x}\}
\end{equation}
for all $x\in\R^n$. By the Courant-Fischer theorem~\cite[Thm. 7.3.10]{HJbook} the singular values of $A$ can be computed as
\begin{equation}\label{this}
\sigma_j(A)= \min_{\dim S=n+1-j} \max_{\substack{x\in S\\ \abs{x}=1}} \abs{Ax},\quad j=1,\dots,n
\end{equation}
where the minimum is taken over all subspaces $S\subset \R^n$ of dimension $n+1-j$.
Using~\eqref{this}  and~\eqref{lambda3} to estimate  $\sigma_j(A^{\lambda})$, we arrive at
\begin{equation}\label{lambda4}
\sqrt{1-\delta^2}\max\{\sigma_j(A),\lambda\}\le
\sigma_j(A^{\lambda})\le 2\max\{\sigma_j(A),\lambda\}.
\end{equation}
For $1 \le j < \ell \le n$ the double inequality~\eqref{lambda4} yields
\[ \frac{\sigma_j(A^\lambda)}{\sigma_\ell (A^\lambda)} \le \frac{2}{\sqrt{1- \delta^2}} \frac{\max \left(\sigma_j (A), \lambda \right)}{\max \left(\sigma_\ell (A), \lambda \right)} \le     \frac{2}{\sqrt{1- \delta^2}} \frac{\sigma_j(A)}{\sigma_\ell (A)}.  \]
This implies~\eqref{lambda1} by virtue of~\eqref{distdef}. Finally~\eqref{lambda2} follows from the left hand side of (\ref{lambda3}).
\end{proof}
 \begin{remark}\label{positive}
 Under the assumptions of Lemma \ref{lambdadist} we have $\det A^\lambda >0$. Indeed, $\det A^\lambda \ne 0$ by~\eqref{lambda2}. Since $\lambda^{-n} \det A^\lambda \to 1$ as $\lambda \to \infty$, the continuity of $\lambda\mapsto \det A^\lambda$ yields $\det A^\lambda >0$.
 \end{remark}

\section{Proofs of  Theorems \ref{ThMain0}, \ref{ThMain}, \ref{ThMain2} and  \ref{ThMain3}}\label{mainproofs}

If a smooth invertible mapping $f$ has $(Df)^{-1} \in L^\infty$, then its inverse $f^{-1}$ is Lipschitz. We prove a similar estimate without assuming injectivity or smoothness of $f$.

\begin{lemma}\label{liminf}
Suppose $f\in W^{1,n}_{\loc} (\Omega, \R^n)$ is a discrete and open mapping such that
\begin{equation}\label{bound}
\norm{(Df)^{-1}} \le M \quad \mbox{ a.e. in }\Omega.
\end{equation}
Then for every $a\in \Omega$ we have
\begin{equation}
\liminf\limits_{x \to a} \frac{|f(x)-f(a)|}{|x-a|}  \ge \frac{1}{2M (i(a,f))^2}>0.
\end{equation}
\end{lemma}
\begin{proof}
Since $f$ is discrete and open it is either sense-preserving or sense-reversing \cite[p. 18]{Ribook}. We may assume that $f$ is sense-preserving. Since $f^{-1}$ does not exist in general we will work with a substitute mapping $g_U$ defined in (\ref{gUreal}) below.
To prove that $g_U$ is Lipschitz we will use Theorem 2.1 in \cite{KO} which requires the following.
\begin{itemize}
\item[(i)]{Condition $(N)$: $f$ preserves sets of measure zero}
\item[(ii)]{Condition $(N^{-1})$:  $f^{-1}(E)$ has measure zero whenever $E$ does}
\item[(iii)]{The branch set   $B_f$   has measure zero}
\end{itemize}
First we observe that the inner distortion of $f$ is locally integrable because
\begin{equation}\label{blah1}
K_I(x,f) =  \norm{(Df(x))^{-1}}^n  J(x,f) \le M^nJ(x,f).
\end{equation}
By Theorem A  \cite{MM} the mapping $f$ satisfies  condition $(N)$.  The validity of $(N^{-1})$  follows from  Theorem 1.2 in \cite{KM}. By virtue of  (\ref{bound})  $J(\cdot ,f)>0$ a.e.,  which implies  (iii) because $Df$ cannot be invertible at  branch points \cite[Lemma I.4.11]{Ribook}.

Let  $U$ be  a normal neighborhood  of $a\in\Omega$. For a sufficiently small  $r>0$ we choose $u\in C^\infty_0(\Omega)$ such that
\begin{equation}\label{gU}
u\ge0, \quad u(a)=0, \quad u \ge 1 \mbox{ on } U \setminus B(a,r), \quad |\nabla u(x)| \le  \frac{2}{r}   \mbox{ for all } x \in \Omega.
\end{equation}
Then we define
\begin{equation}\label{gUreal}
g_U(y)= \sum_{x\in f^{-1}(y) \cap U} i (x,f) u(x)
\end{equation}
for $y\in f(U)$. We claim that $|\nabla g_U|$ is bounded a.e. in $f(U)$.

Fix $y\in f(U)$ and $\rho< \dist(y, \partial f(U))$.   By Lemma I.4.7 in \cite{Ribook} the set $f^{-1}(B(y, \rho)) \cap U$ is a finite union of normal domains $V_1, \dots, V_m$ where $m \le N(f,U)=i(a,f)$. It is clear that $g_U(z) = \sum_{i=1}^m g_{V_i} (z)$ for all $z\in B(y, \rho)$.

By Theorem 2.1 in \cite{KO} $g_{V_i} \in W^{1,n} (B(y, \rho)  )$ for all $i\in \{1, \dots, m\}$ with the estimate
\begin{equation}\label{blah}
 \int_{B(y, \rho)} |\nabla g_{V_i}(z)|^n \, \dtext z \le N(f,V_i)^{n-1} \int_{V_i} |\nabla u (x)|^n K_I(x,f)\,   \dtext x. 
 \end{equation}
Jensen's inequality  yields
\[\abs{\nabla g_U(z)}^n \le m^{n-1} \sum_{i=1}^m \abs{g_{V_i}(z)}^n. \]
Combining this with (\ref{blah}) and summing over $i$   we obtain
\[  \int_{B(y, \rho)} |\nabla g_{U}(z)|^n \, \dtext z \le  N(f,U)^{2n-2}  \int_{f^{-1}(B(y, \rho))}   |\nabla u (x)|^n   K_I(x,f)\,   \dtext x.  \]
Next we use  (\ref{gU}) to estimate $\abs{\nabla u}$ from above:
\[ \int_{B(y, \rho)} |\nabla g_{U}(z)|^n \, \dtext z \le  N(f,U)^{2n-2} (2/r)^n \int_{f^{-1}(B(y, \rho))}      K_I(x,f)\,   \dtext x.  \]
We estimate $K_I(\cdot, f)$ by  (\ref{blah1}).
\[  \int_{B(y, \rho)} |\nabla g_{U}(z)|^n \, \dtext z \le   N(f,U)^{2n-2} (2M/r)^n  \int_{f^{-1}(B(y, \rho))}  J(x,f)\,   \dtext x. \]
Applying the change of variable inequality, Lemma I.4.11 in \cite{Ribook}, we obtain
\[  \int_{B(y, \rho)} |\nabla g_{U}(z)|^n \, \dtext z \le   N(f,U)^{2n-1} (2M/r)^n  |B(y, \rho)|. \]
Dividing by $|B(y, \rho)|$ and letting $\rho\to0$ we conclude that  $|\nabla g_U|\le C/r$  a.e. in $f(U)$, where $C= 2MN(f,U)^2$.

We have established that $\abs{g_U(z)-g_U(w)}\le (C/r)\abs{z-w}$ for all $z,w\in B(f(a), d)$, where $d=\dist(f(a),\partial f(U))$.
It also follows from \eqref{gU} that $g_U(f(a))=0$. Therefore $g_U(y)<1$ whenever   $\abs{y-f(a)} < r/C$. By the definition of $g_U$ this implies
$f^{-1}(y)\cap U \subset B(a,r)$ for all $y\in B(f(a), r/C)$. This completes the proof.
\end{proof}

In the proof of the following lemma we use an idea of Kirchheim and Sz\'ekelyhidi~\cite{KS}.

\begin{lemma}\label{lochom}
Under the assumptions of Theorem \ref{ThMain} the mapping $f^\lambda(x)=f(x)+ \lambda x$ is a local homeomorphism for all $\lambda >0$.
\end{lemma}
\begin{proof}
Lemma \ref{lambdadist} implies  $K_O(\cdot ,f^\lambda) \in L^p_{\loc}(\Omega)$ where $p=1$ if $n=2$ and $p>n-1$ if $n \ge 3$.  Also, $f^\lambda$ is nonconstant, for otherwise we would have $Df=- \lambda I \notin \mathcal M_n(\delta)$. Therefore, $f^\lambda$ is discrete and open by Theorem \ref{discreteopen}.

Suppose to the contrary that $f^\lambda$ is not a local homeomorphism.  Without loss of generality $0\in B_{f^\lambda}$ and $f(0)=0$.
Since $f$ is sense-preserving, $i(0,f)\ge 2$. Define
\[\Lambda = \sup \{\lambda \colon 0\in B_{f^\lambda} \}.  \]
Recall that  the topological index is  upper semicontinuous; that is,
\begin{equation}\label{uppersemi}
i(x,g) \ge \limsup_{n \to \infty} i(x,g_n)
\end{equation}
provided $g_n \to g$ locally uniformly, and $g_n$ and $g$ are discrete and open.

Since
\[\lim_{\lambda \to \infty } \frac{f^\lambda (x)}{\lambda +1} =x\]
locally uniformly it follows that $0 \notin B_{f^\lambda}$ for sufficiently large $\lambda$. Therefore $\Lambda < \infty$.
By \eqref{uppersemi} we have $0\in B_{f^\Lambda }$. It follows from~\eqref{lambda2} that
\begin{equation}
\norm{ (Df^\Lambda )^{-1}} \le \frac{1}{\Lambda \sqrt{1- (\delta\wedge 0)^2 }} \nonumber
\end{equation}
a.e in $\Omega$. Applying Lemma \ref{liminf} to $f^{\Lambda}$ we obtain
\begin{equation}\label{13}
\liminf\limits_{x \to 0}  \frac{\abs{f^\Lambda (x)}}{|x|} >0.
\end{equation}
We will obtain a contradiction by showing that  $0 \in B_{f^\lambda}$ whenever $|\lambda - \Lambda|$ is sufficiently small.  By (\ref{13}) there exist  $\epsilon >0$ and $r_0 >0$  such that $ \abs{f^\Lambda (x)} \ge \epsilon |x|$ whenever  $0<|x| <r_0$. Fix $\lambda>0$ such that $\abs{\lambda - \Lambda} < \epsilon$.  Since $\abs{f^\lambda (x)- f^\Lambda (x)} = |\lambda - \Lambda| |x|$, the homotopy invariance of  topological degree yields
\[ \mu ( 0, f^\lambda , B(0,r) ) = \mu  ( 0, f^\Lambda , B(0,r)  )  \]
for all $0<r<r_0$. Therefore, $i(0,f^\lambda)=i(0,f^\Lambda)\ge 2$, which means $0\in B_{f^{\Lambda}}$.
This is the desired contradiction.
\end{proof}

\begin{proof}[Proof of Theorem \ref{ThMain}]
By Lemma \ref{lochom} $f^\lambda$ is a local homeomorphism for all $\lambda>0$.  Also  $f$ is discrete and open by Theorem \ref{discreteopen}.  Applying  Proposition \ref{PrVa}  we  conclude that $f$ is a local homeomorphism.
\end{proof}

\begin{proof}[Proof of Theorem \ref{ThMain2}]
As in the proof of Theorem \ref{ThMain} we have that $f^\lambda$ is a local homeomorphism for all $\lambda>0$.
It follows from~\eqref{lambda2} that
\begin{equation}
\norm{ (Df^\lambda )^{-1}} \le \frac{1}{\lambda \sqrt{1- (\delta\wedge 0)^2 }} \nonumber
\end{equation}
a.e. in $\R^n$.
Now Lemma \ref{liminf} yields
\begin{equation}
\liminf\limits_{x \to a} \frac{|f(x)-f(a)|}{|x-a|}  \ge \frac{\lambda \sqrt{1-(\delta \wedge 0)^2}}{2}>0
\end{equation}
for all $a\in \R^n$. By Theorem \ref{john} $f^\lambda$ is  a homeomorphism.  Since $f$ is discrete and open by Theorem \ref{discreteopen} we can apply  Proposition \ref{PrVa} and conclude that $f$ is a   homeomorphism.
\end{proof}

Theorem~\ref{ThMain0} is an immediate consequence of Theorems~\ref{ThMain} and ~\ref{ThMain2}.

\begin{proof}[Proof of Theorem \ref{ThMain3}]
We will apply Theorems \ref{ThMain} or \ref{ThMain2}, as appropriate, to $f^\lambda(x)=f(x)+ \lambda x$ for $\lambda>0$. In order to do so we must show that $f^\lambda$ is a mapping of finite distortion with sufficient degree of integrability of $K_O(\cdot,  f)$. First, $J (\cdot, f^\lambda ) >0$ a.e. by Remark \ref{positive}. Second, using~\eqref{lambda2}  and \eqref{distdef} we obtain
\begin{equation}\label{instead}
K_O(x,f^\lambda ) \le \frac{\norm{Df^\lambda (x)}^{n-1}}{\lambda^{n-1} \left( 1- (\delta \wedge 0)^2 \right)^\frac{n-1}{2}}
\end{equation}
a.e. in $\Omega$. Therefore $K_O(\cdot ,f^\lambda)\in L^p_{\loc} (\Omega)$,  where $p=2$ if $n=2$ and  $p>n-1$ if $n \ge 3$. It remains to apply Theorems \ref{ThMain} or \ref{ThMain2} and pass to the limit $\lambda \to 0$ using Proposition \ref{PrVa}.
\end{proof}

We conclude  this section with two remarks concerning integration of differential inclusions.

\begin{remark}\label{monotone}
The case $\delta \ge0$ of Theorem \ref{ThMain3} only requires $p=1$. Indeed,  integrating $Df$ along a.e. line segment and using the assumption  $Df \in \mathcal M_n(\delta)$ we find that
\begin{equation}\label{deltamon}
 \inn{f(a)-f(b), a-b} \ge \delta \abs{f(a)-f(b)}  \abs{a-b}
 \end{equation}
whenever the line segment $[a,b]$ is contained in $\Omega$. The inequality is true for all such line segments since  $f$ is continuous.
For any $\lambda >0$, $f^\lambda$ is a local homeomorphism because
\[ \inn{f^\lambda(a)-f^\lambda(b), a-b} \ge \lambda \abs{a-b}^2\]
provided that $[a,b] \subset \Omega$. Passing to the limit $\lambda \to 0$ and using Proposition \ref{PrVa} we conclude that $f$ is a local homeomorphism (global homeomorphism if $\Omega$ is convex).
\end{remark}

 \begin{remark}\label{notmonotone}
When $-1 < \delta <0$, the differential inclusion $Df \in \mathcal M_n(\delta)$  cannot be integrated to yield~\eqref{deltamon}. A counterexample is given by the complex function $f(z)=z^{5/2}$   defined in the right half-plane $\Omega = \{z \in \C \colon \re z > 0\}$. Indeed, for $2\pi/5 < \theta < \pi /2$ the vectors $f(e^{i\theta})- f(e^{-i\theta})$ and $e^{i \theta}- e^{-i \theta}$ point in exactly opposite directions. This mapping  also fails to be injective, which shows that the global injectivity  part of  Theorems   \ref{ThMain0},  \ref{ThMain2} and  \ref{ThMain3}   is not true for general convex domains.
\end{remark}

\section{Planar case: proof of Corollary~\ref{QRThm}}\label{planar}

In the planar case  $n=2$ the quasiregularity assumption can be expressed in terms of the complex derivatives  $f_z$ and $f_{\bar z}$ as follows. A mapping  $f\in W_{\rm loc}^{1,2}(\Omega; \C)$ is $K$-quasiregular if and only if $\abs{f_{\bar z}}\leqslant k\abs{f_{z}}$ a.e. in $\Omega$, where $K=(1+k)/(1-k)$. We also need to translate the condition $Df\in \M_2(\delta)$ into the complex language.

\begin{lemma}\label{dcom}
Let $f \colon \Omega \to \mathbb C$ be a mapping, and fix $\delta\in (-1,1)$. At each point of differentiability of $f$ the following three conditions are equivalent:
\begin{enumerate}[(i)]
\item\label{dcomi} $Df \in \mathcal M_2(\delta)$
\item\label{dcom0} $\abs{\arg f_z}+\arcsin \abs{f_{\bar z}/f_z}\le \arccos \delta$
\item\label{dcomiii} either
\begin{equation}\label{dcom1}
\abs{f_{\bar z}}+\delta\abs{\im f_z}\le\sqrt{1-\delta^2}\re f_z,\quad \textrm{or}
\end{equation}
\begin{equation}\label{dcom2}
\abs{f_{\bar z}}\le \abs{f_z}\le \frac{1}{\sqrt{1-\delta^2}} \re f_z.
\end{equation}
\end{enumerate}
\end{lemma}
\begin{proof} Fix a point $z\in\Omega$ where $f$ is differentiable and let $\alpha=f_z(z)$, $\beta=f_{\bar z}(z)$.
Condition~\eqref{dcomi} is readily seen to be equivalent to
\begin{equation}\label{dcom3}
\re(\alpha+\beta \bar z/z)\ge \delta\abs{\alpha+\beta \bar z/z},\quad z\ne 0.
\end{equation}
In geometric terms~\eqref{dcom3} means that the disk with center $\alpha$ and radius $\abs{\beta}$ is contained in the closed
sector $\{\zeta\in\C\colon \abs{\arg\zeta}\le \cos^{-1}\delta\}$. The latter condition is equivalent to~\eqref{dcom0} because the half-lines
$\arg \zeta=\arg \alpha\pm \arcsin \abs{\beta/\alpha}$ are tangent to the disk.

Next we prove that~\eqref{dcom0} implies~\eqref{dcomiii}. Suppose that~\eqref{dcom0} holds. An immediate consequence is
$\abs{\beta}\le \abs{\alpha}$, i.e. the first part of~\eqref{dcom2}. If also the second inequality in~\eqref{dcom2} holds, then
we have~\eqref{dcomiii}. Thus we may assume that the second inequality in~\eqref{dcom2} fails, which means
$\abs{\arg \alpha}> \arccos\sqrt{1-\delta^2}$. We can rewrite~\eqref{dcom0} as
\begin{equation}\label{dcom4}
\arcsin \abs{\beta/\alpha }\le \arccos \delta-\abs{\arg \alpha}
\end{equation}
where the right-hand side is less than $\arccos\delta-\arccos \sqrt{1-\delta^2}\le \pi/2$. Therefore, we can apply the sine function
to both sides of~\eqref{dcom4} and obtain
\begin{equation}\label{dcom5}
\abs{\beta/\alpha }\le \sqrt{1-\delta^2}\cos \arg \alpha -\delta \sin \abs{\arg\alpha},
\end{equation}
which after multiplication by $\abs{\alpha}$ becomes
\[
\abs{\beta}\le \sqrt{1-\delta^2}\re \alpha -\delta \abs{\im\alpha}.
\]
This completes the proof of~\eqref{dcomiii}.

Finally, we will suppose~\eqref{dcomiii} and prove~\eqref{dcom0}. If~\eqref{dcom1} holds, then
\[
\abs{\beta/\alpha }\le \sqrt{1-\delta^2}\cos \arg \alpha -\delta \sin \abs{\arg\alpha}=\sin(\arccos \delta-\abs{\arg \alpha}),
\]
which implies~\eqref{dcom0}. If~\eqref{dcom2} holds, then
\[
\abs{\arg \alpha} \le \arccos\sqrt{1-\delta^2}\le \arccos\delta-\frac{\pi}{2},
\]
and~\eqref{dcom0} follows again.
\end{proof}

\begin{remark}\label{pointless} In the special case $\delta\ge 0$ one can simplify Lemma~\ref{dcom} by removing~\eqref{dcom2}. This fact can be inferred from Theorem~3.11.6 in~\cite{AIMbook} or, in a less precise form, from Lemma~10 in~\cite{K}.
\end{remark}

\begin{proof}[Proof of Corollary \ref{QRThm}]  Recall that
\begin{equation}\label{tau}
\tau_K= \frac{2 \sqrt{K}}{K+1}= \sqrt{1-k^2},\quad K=\frac{1+k}{1-k}.
\end{equation}
We may assume that $0<\tau<\tau_K$.
   Set $\delta = \cos \left(\pi- \arccos \tau+ \arcsin k \right)$. Since $\tau < \sqrt{1-k^2}$, it follows that
\[\arccos \tau > \arccos (\sqrt{1-k^2}  ) =  \arcsin k. \]
Therefore, $-\pi/2\le \arcsin k - \arccos \tau<0$, which implies
$\delta >-1$.  For a.e. $z\in \Omega$ we have
\begin{eqnarray*}
\abs{\arg f_z} + \arcsin \abs{f_{\bar z}/f_z} &\le& \arccos(-\tau) + \arcsin k \\
&=& \pi - \arccos \tau + \arcsin k = \arccos \delta.
\end{eqnarray*}
Now Lemma \ref{dcom} yields  $Df \in \mathcal M_2 (\delta)$ a.e. in $\Omega$. By Theorem \ref{ThMain} $f$ is a local homeomorphism. If $\Omega = \C$, then $f$ is a homeomorphism by Theorem \ref{ThMain2}.
\end{proof}

\section{Failure of invertibility}\label{example}

First we show the sharpness of Corollary  \ref{QRThm} by constructing a mapping as in Example~\ref{branchex}. That is, we  exhibit  a nonconstant  mapping  $f \colon \C \to \C$ such that
\begin{equation}\label{exam2}
 \re f_z\ge -\sqrt{1-k^2}\abs{f_z}\quad \text{a.e. in }\C
\end{equation}
and $f$ has a branch point at $0$. Here $\sqrt{1-k^2}=\tau_K$, see~\eqref{tau}.
The construction is split into two cases depending on the value of $k$.

{\bf Case 1.} $0 \le k \le 1/\sqrt{2}$. In this case we will construct a $K$-quasiregular mapping $f$ with branch point at $0$ such that
\begin{equation}\label{exam1}
\abs{\re f_z}\le \sqrt{1-k^2}\abs{f_z}\quad \text{a.e. in }\C.
\end{equation}
Let $a= \sqrt{1-k^2}+ik$, $b= -ik$.  First we define a quasiregular mapping from the first quadrant $Q_1= \{ x+iy \in \mathbb C \colon x,y \ge 0  \}$
 onto the upper half plane by the formula
 \[f(z)= az^2 + b \bar z^2. \]
Its complex derivatives are
\[f_z= 2az , \qquad f_{\bar z} = 2b\bar z. \]
Note that $\abs{f_{\bar z}} = k \abs{f_z}$. Also the image of $Q_1$ under the mapping  $f_z$ is the rotated quadrant $R_1= \{ w \colon \arg a \le \arg w \le \pi/2 + \arg a \}$. Since $\arg a \in [0, \pi/4]$, for all $w\in R_1$ we have
\[ \frac{\abs{\re w}}{|w|} = \abs{\cos \arg w} \le \cos \arg a   =  \sqrt{1-k^2}. \]
Therefore,~\eqref{exam1} holds in  the interior of $Q_1$.

By reflection about the real axis we extend $f$ to the fourth quadrant $Q_4= \{ x+iy \in \mathbb C \colon x \ge 0, \; y \le 0 \}$; that is, we set $f(z)= \overline{f(\bar z)}$ if $z\in Q_4$. For $z$ in the interior of $Q_4$ we have
 $f_z(z)= \overline{f_z (\bar z)}$ and therefore  \eqref{exam1} holds again.

Finally we extend $f$ to the left halfplane as an even function; that is, $f(z)= f(-z)$. It is easy to see that $f$ is quasiregular in $\C$
and~\eqref{exam1} is satisfied everywhere in $\C$ except for the coordinate axes.

{\bf Case 2.} $1/\sqrt{2} < k <1$. Following Example 5.1 in \cite{IKO} we introduce the parameters $\epsilon, \delta \in (0,1)$ so that
\[k = \frac{1}{\sqrt{1+\epsilon^2}}  \quad \text{and} \quad \delta = \frac{\epsilon}{2 + \sqrt{4- \epsilon^2}}. \]
 For $\im z\ge 0$ we define $f(z)$ by the formula
\begin{equation}
f(z)= \begin{cases}  \displaystyle \frac{2z^2}{|z| \sqrt{1+ \delta^2}},\quad &\textnormal{if } \re z \geqslant - \delta \im z \\ (i-\epsilon )z - i \bar z,\quad &  \textnormal{if } \re z \leqslant - \delta \im z
  \end{cases}
\end{equation}
and then extend it to the lower halfplane by reflection $f(z)=\overline{f(\bar z)}$.  In \cite{IKO} it is proved that
\[\abs{f_{\bar z}} \le k \abs {f_z} \quad \text{and} \quad  \re f_z \ge - \epsilon \abs{\im f_z} \quad \text{a.e. in }\C. \]
It remains to show that~\eqref{exam2} holds.
If $\re f_z \ge 0$ there is nothing to prove. Otherwise we have $ \abs{ \re f_z} \le  \epsilon \abs{\im f_z}$ which implies
\[\frac{(\re f_z)^2}{\abs{f_z}^2} = \frac{(\re f_z)^2}{(\re f_z)^2 + (\im f_z)^2}  \le \frac{\epsilon^2}{ \epsilon^2 +1}= 1-k^2.\]
This proves \eqref{exam2}.
\qed

\begin{example}\label{ball}
Let $\Omega$ be the unit ball $B(0,1) \subset \R^n$, $n \ge 2$. For any $\delta \in (-1,0)$  there exists a Lipschitz mapping $f \colon \Omega \to \R^n$ such that $K_O(\cdot, f ) \in L^q(\Omega)$ for all $q<n-1$ and $Df \in \M_n(\delta)$ a.e. but $f$ is not a local homeomorphism.
\end{example}
Fix $\delta \in (-1, 0)$ and set  $\epsilon = - \delta $.
Following~\cite{Ba}, define $f\colon \Omega \to \R^n$ by the rule
\[ f(x)=  (x_1,x_2,\dots  ,x_{n-1},  \epsilon  s(x)x_n)  \]
where  $s(x)=\sqrt{x_1^2+x_2^2+\dots+x^2_{n-1}}$. The construction of Example~1 in~\cite{Ba} shows that $f$ is Lipschitz and $J(x,f)>0$ for
almost every $x \in \Omega$. So $f$ is a mapping of finite distortion and a direct computation
shows that
$$\int_{\Omega} (K_O(x,f))^q\, dx\leq  C_n \int_\Omega s(x)^{-q }\, dx <\infty. $$

 On the other hand, $f(x)=0$ whenever $s(x) = 0$, which implies that $f$ is not a local homeomorphism.

It remains to show that $Df \in \M_n(\delta)$ a.e. in $\Omega$. Fix $x\in \Omega$ such that $s(x) \ne 0$ and  denote $A=Df(x)$.
The entries of $A$ are as follows: for $i=1, \dots, n-1$
\[ a_{ij} = \begin{cases} 1 & \; \text{ if } j=i \\
0 & \; \text{ if } j\ne i
\\\end{cases} \]
and
\[ a_{nj} = \begin{cases} \frac{\epsilon x_jx_n}{s(x)}  & \; \text{ if } j=1, \dots , n-1\\
\epsilon s(x)  & \; \text{ if } j=n. \end{cases} \]
Let $\xi \in \R^n$ be a unit vector. We must prove that
\begin{equation}\label{19}
 \langle A \xi , \xi \rangle \ge - \epsilon |A\xi|   .
\end{equation}
Here
\[A \xi = \left( \xi_1, \dots , \xi_{n-1} , t    \right) ,   \qquad t= a_{n1} \xi _1 + \dots + a_{n \, n-1} \xi_{n-1} + a_{nn} \xi_n . \]
If $t \xi_n \ge 0$, then $ \langle A \xi , \xi \rangle \ge 0$ and we are done.  Suppose  $t \xi_n < 0$. Since $a_{nn} \ge 0$, it follows that
\begin{eqnarray*}
-t \xi_n &\le &  \abs {\xi_n} \abs{ a_{n1} \xi _1 + \dots + a_{n \, n-1} \xi_{n-1} }  \\
&\le& \abs{\xi_n} \Bigg(\sum_{j=1}^{n-1} a^2_{nj}\Bigg)^{1/2}  \Bigg(\sum_{j=1}^{n-1} \xi^2_{j}\Bigg)^{1/2} = \abs{\xi_n} \epsilon \abs{x_n} s(\xi).
\end{eqnarray*}
Dividing by $|\xi_n|$ we obtain
\[ |t| \le   \epsilon \abs{x_n} s(\xi) \le \epsilon s(\xi).\]
Therefore,
\[ \langle A \xi , \xi \rangle = s(\xi)^2 + t \xi_n \ge - \epsilon s(\xi) \abs{\xi_n} \ge - \epsilon s(\xi)  . \]
Since $\abs{A \xi } \ge s(\xi)$,  inequality (\ref{19}) follows. \qed

\section{Concluding remarks}\label{concluding}
Theorems \ref{ThMain}, \ref{ThMain2} and \ref{ThMain3} could be strengthened if the well-known Iwaniec-Martin-\v Sver\'ak conjecture is true.
\begin{conjecture}  \cite[Conj. 6.5.1]{IMbook}
If $f \in W^{1,n}_{\loc}(\Omega, \R^n)$  and $K_I (\cdot, f)\in L^1_{\loc}(\Omega)$ then $f$ is either constant or both discrete and open.
\end{conjecture}
Namely, in Theorems \ref{ThMain} and \ref{ThMain2} the requirement on the distortion function could be replaced by  $K_I (\cdot, f)\in L^1_{\loc}(\Omega)$. Indeed, this requirement is only needed to ensure that  $f$ and $f^\lambda$ are discrete and open.

In Theorem \ref{ThMain3} it would suffice to have $p=n$ for all $n \ge 2$. Then instead of~\eqref{instead} we would use the estimate
\[ K_I(x,f^\lambda ) \le \frac{\norm{Df^\lambda (x)}^{n-1}}{\lambda^{n-1} \left( 1- (\delta \wedge 0)^2 \right)^\frac{n-1}{2}} \]
which implies $K_I (\cdot, f)\in L^\frac{n}{n-1}_{\loc}(\Omega)$. The requirement $f\in W^{1,n}_{\loc}(\Omega, \R^n)$ remains  essential in view of the examples in \cite{KKM}. Observe that for this improvement of Theorem \ref{ThMain3} it would suffice to verify the  Iwaniec-Martin-\v Sver\'ak conjecture above the critical exponent $1$; that is, for $K_I (\cdot, f)\in L^q_{\loc}$ with $q>1$. It seems likely that this case of the conjecture should be more tractable. Indeed, in terms of the outer distortion the critical exponent is $n-1$ and  discreteness and openness have been proved  above the critical exponent  by Manfredi and Villamor~\cite{MV}.
In the critical case the conjecture has been verified under the additional assumption that $f$ is  quasi-light~\cite{HM, Ra}.

\section*{Acknowledgements}
We thank Tadeusz Iwaniec and Jarmo J\"a\"askel\"ainen for stimulating discussions on the subject of this paper. We are very grateful to Jussi V\"ais\"al\"a for providing us with a proof of Proposition~\ref{PrVa}.

\bibliographystyle{amsplain}

\end{document}